\documentclass[11pt,reqno]{amsart}

\usepackage[margin=1.15in]{geometry}
\usepackage[utf8]{inputenc}
\usepackage[T1]{fontenc}
\usepackage{amsmath,amssymb,amsfonts,mathtools}
\usepackage{enumitem}
\usepackage[colorlinks=true,linkcolor=blue,citecolor=blue,urlcolor=blue]{hyperref}
\usepackage[nameinlink,noabbrev]{cleveref}
\usepackage{bbm}

\title{Classification of blow-ups for the Alt--Phillips problem in three dimensions}
\author{Xavier Fern\'andez-Real}
\address{EPFL SB, Station 8, 1015 Lausanne, Switzerland}
\email{xavier.fernandez-real@epfl.ch}
\thanks{The author was supported by the Swiss National Science Foundation (SNF grant PZ00P2\_208930), by the Swiss State Secretariat for Education, Research and Innovation (SERI) under contract number MB22.00034, and by the AEI project PID2021-125021NA-I00 (Spain).}
\date{}


\newtheorem{thm}{Theorem}[section]
\newtheorem{lem}[thm]{Lemma}
\newtheorem{cor}[thm]{Corollary}
\newtheorem{prop}[thm]{Proposition}

\theoremstyle{definition}
\newtheorem{defn}[thm]{Definition}

\theoremstyle{remark}
\newtheorem{remark}[thm]{Remark}

\crefname{thm}{Theorem}{Theorems}
\crefname{lem}{Lemma}{Lemmas}
\crefname{cor}{Corollary}{Corollaries}
\crefname{prop}{Proposition}{Propositions}
\crefname{claim}{Claim}{Claims}
\crefname{defn}{Definition}{Definitions}
\crefname{notation}{Notation}{Notations}
\crefname{remark}{Remark}{Remarks}

\numberwithin{equation}{section}

\newcommand{\R}{\mathbb R}
\newcommand{\Sn}{\mathbb S^2}
\newcommand{\Om}{\Omega}

\newcommand{\apalpha}{\alpha}

\newcommand{\lamax}{\lambda_{\max}}
\newcommand{\lamin}{\lambda_{\min}}

\newcommand{\Jh}{\mathcal J_{\!h}}
\newcommand{\Lin}{\mathcal L}
\newcommand{\HH}{\mathsf A}
\newcommand{\grad}{\nabla}
\newcommand{\tr}{\operatorname{tr}}
\newcommand{\dist}{\operatorname{dist}}
\newcommand{\supp}{\operatorname{supp}}
\newcommand{\eps}{\varepsilon}

\begin{document}

\begin{abstract}
We prove a flatness theorem for classical stable homogeneous solutions of the $\gamma$-Alt--Phillips free boundary problem in three dimensions in the range $0<\gamma\le2/3$, where $\gamma$ is the exponent in the energy density $|\nabla u|^2+u^\gamma$. In particular, this implies full regularity of the free boundary for minimizers of the corresponding Alt--Phillips energy in dimension $3$.
\end{abstract}

\maketitle

\tableofcontents

\section{Introduction}\label{sec:introduction}
Free boundary regularity is closely tied to the classification of homogeneous solutions.  In one-phase variational problems, blow-ups at singular points are cones, and stable or minimizing cones are the natural obstructions to flatness and higher regularity.  This perspective goes back to the Alt--Caffarelli one-phase or Bernoulli  problem \cite{AltCaffarelli1981ExistenceRegularityOnePhase} and to subsequent monotonicity and blow-up methods for elliptic free boundaries (see also \cite{Weiss1998PartialRegularityEllipticFreeBoundary}, and the expository works \cite{Vel23, CS05, Fer26}). For the Bernoulli problem, stable-cone rigidity in low dimension was developed in particular by Jerison--Savin \cite{JerisonSavin2015StableConesOnePhase} (see also, \cite{CaffarelliJerisonKenig2004GlobalMinimizersFreeBoundary, DeSilva2009SingularEnergyMinimizing}).

The present paper concerns the semilinear one-phase problem introduced by Alt--Phillips \cite{Phillips1983MinimizationFreeBoundary, AltPhillips1986FreeBoundarySemilinear}.  For $0<\gamma<1$, the functional we consider has the form
\[
  \mathcal E_\gamma(u,D) =\int_{  D}\left(|\grad u|^2+u^\gamma\mathbbm{1}_{\{u > 0\}}\right)\,dx.
\]
In the smooth free-boundary regime, its Euler--Lagrange system consists of a semilinear equation in $\{u > 0\}$ that is singular at the zero phase  ($\Delta u = \frac{\gamma}{2} u^{\gamma-1}$ in $\{u > 0\}$)  and a degenerate Neumann condition on the free boundary ($|\nabla u| = 0 $ on $\partial \{u > 0\}$); the precise normalization used below is recalled in Section~\ref{sec:preliminaries}.  The natural scaling exponent is
\[
  \beta=\frac{2}{2-\gamma},
\]
so Alt--Phillips blow-ups are $\beta$-homogeneous.  It is therefore convenient to pass to the one-homogeneous variable
\[
  v=\beta u^{1/\beta}, \qquad u=\left(\frac v\beta\right)^\beta .
\]
If
\[
  \apalpha=\beta\gamma=\frac{2\gamma}{2-\gamma},
\]
then the transformed profile satisfies
\[
  \Delta v=\frac{\apalpha}{2}\frac{1-|\grad v|^2}{v} \quad\text{in}\quad \{v>0\}, \qquad\qquad   |\grad v|=1 \quad\text{on}\quad \partial\{v > 0\}.
\]
The range considered in this paper,
\[
  0<\gamma\le \frac23,
\]
is exactly the range $0<\apalpha\le 1$.  In this regime we prove that stable three-dimensional homogeneous classical cones are necessarily flat. This is the first classification result for the Alt--Phillips problem in dimension 3 (whereas dimension 2 is already known, \cite{AltPhillips1986FreeBoundarySemilinear}).

\begin{thm}\label{thm:stable-homogeneous-cone-flatness}
Let $0<\gamma\le 2/3$, $\beta=\tfrac{2}{2-\gamma}$, and let $u$ be a classical stable homogeneous solution to the $\gamma$-Alt--Phillips problem in $\R^3$ (see Definitions~\ref{def:classical-homogeneous-ap-cone} and \ref{def:stability}). Then, there exists $e\in\Sn$ such that
\[
  u(x)=\beta^{-\beta}(x\cdot e)_+^\beta .
\]
\end{thm}

Thus, under the hypotheses of the theorem, the only stable cone is the half-space profile.  The classical assumption corresponds to regularity away from the vertex; in particular, it excludes the two-sided profile $|x\cdot e|$ from the one-phase class. Also, the restriction $0<\gamma\le 2/3$ is consistent with recent evidence for the special role of the threshold $\gamma=\tfrac23$ \cite{KarakhanyanSanzPerela2026StableConesAltPhillips,AllenKriventsovShahgholian2025OscillatingSingularTerms,RestrepoRosOton2025CinftyRegularitySemilinear}; moreover, the counterexamples in \cite{SavinYu2025StableMinimizingConesAltPhillips} for $\gamma$ close to $1$ show that one cannot expect such a statement throughout the whole range $0<\gamma<1$. Notice, for instance, that the problem is, in a way, \emph{less} convex for this range of $\gamma$ (see \eqref{eq:venergy}, where $0<\apalpha < 1$). 

As a classical consequence, we obtain the regularity of free boundaries:
\begin{cor}
\label{cor:free-boundary-regularity}
Let $n \ge 2$ and $0 < \gamma \le 2/3$. Let $ u$ be a minimizer of $\mathcal E_\gamma(u, B_1)$ with $B_1\subset \R^n$. Then, if $n \le 3$, the free boundary $\partial\{u > 0\}$ is a smooth $C^\infty$ manifold in $B_1$. Moreover, if $n \ge 4$, the set of singular free boundary points satisfies  ${\rm dim}_{\mathcal{H}}({\rm Sing}(u))\le n - 4$. 
\end{cor}

We remark that thanks to our analysis we are also able to characterize the range of $\gamma$ for which cones with full support are not admissible stable solutions; see Proposition~\ref{prop:full-support-exclusion}.

\medskip

The Alt--Phillips problem has recently been revisited from several related viewpoints.  De Silva--Savin developed a viscosity framework for degenerate one-phase problems that includes the transformed Alt--Phillips equation \cite{DeSilvaSavin2021DegenerateOnePhase}. Generic regularity was established by the author and Yu in \cite{FernandezRealYu2023GenericPropertiesFreeBoundary}.  Restrepo--Ros-Oton proved smoothness of regular free boundaries for semilinear problems of this type \cite{RestrepoRosOton2025CinftyRegularitySemilinear} (see also \cite{carducci2026smoothness}).  The stable-cone formulation and the transformed stability inequality used here are due to Karakhanyan--Sanz-Perela \cite{KarakhanyanSanzPerela2026StableConesAltPhillips}; related stable and minimizing Alt--Phillips cones have also been studied by Savin--Yu \cite{SavinYu2025StableMinimizingConesAltPhillips} (see also \cite{SavinYu2026ConcentrationConesAltPhillips}).  These results place the Alt--Phillips theory in parallel with the classical Bernoulli theory, while also showing that the singular weight in the transformed equation creates new parameter-dependent phenomena.

\subsection{Ideas of the proof}
The stability inequality is first reduced to a sharp radial Hardy obstruction.  If a nonnegative homogeneous test function $G$, of degree $-p$, satisfies
\[
  \Lin G\ge \mu\frac{G}{r^2},
\]
 in the positive phase (where $\mathcal{L}$ is the linearized operator, \eqref{eq:linearized-operator}), then stability forces
\[
  \mu\le \lambda_{\apalpha,p}, \qquad \lambda_{\apalpha,p}=\frac{(1+\apalpha-2p)^2}{4}.
\]
This reduction immediately rules out full-support stable cones in certain ranges (see Proposition~\ref{prop:full-support-exclusion}).  It is the only global input needed later; the rest of the argument is a pointwise and distributional analysis on a classical cone.

For classical cones, the main object is the scale-invariant Hessian
\[
  \HH=rD^2v.
\]
The flat half-space solution is characterized by $\HH\equiv0$.  We first obtain the global gradient bound $|\grad v|\le1$, and then prove positivity of the mean curvature of the regular free boundary.  In the interior, a collection of Euclidean Hessian identities for the transformed equation yields a pinching inequality for the eigenvalues of $\HH|_{x^\perp}$.  Writing
\[
  \lambda_M=\lamax(\HH|_{x^\perp}), \qquad \lambda_m=\lamin(\HH|_{x^\perp}), \qquad T=\lambda_M+\lambda_m,
\]
the pinching takes the form
\[
  \lambda_M+(1+\apalpha)\lambda_m>0, \qquad\text{equivalently}\qquad T> \frac{\apalpha}{1+\apalpha}\lambda_M.
\]
In a way, this estimate is a sharper version, in the singular Alt--Phillips setting, of the curvature positivity used in the Bernoulli cone argument.

The final step is a Jerison--Savin type test based on the largest Euclidean Hessian eigenvalue
\[
  \Lambda=\lamax(D^2v)=\frac{\lambda_M}{r}.
\]
For suitable $p$, the function $G=\Lambda^p$ satisfies the distributional inequality
\[
  \Lin G\ge p(1+p)\frac{G}{r^2}.
\]
For $0<\apalpha<1$, the sharp Hardy constant $\lambda_\apalpha=(1+\apalpha)^2/4$ is strictly smaller than one, and one can choose
\[
  \frac{\lambda_\apalpha}{2+\apalpha}<p<\frac{1}{2+\apalpha}.
\]
For this choice, $p(1+p)>\lambda_{\apalpha,p}$, contradicting the Hardy obstruction unless $\HH\equiv0$. The endpoint case $\alpha = 1$ is done separately. 

The paper is organized as follows.  In Section~\ref{sec:preliminaries} we fix the normalization, record the transformed equation and stability inequality, and introduce the homogeneous Hessian notation.  In Section~\ref{sec:reduction}, we prove the Hardy reduction and use it to exclude full-support cones.  The positivity of the free-boundary mean curvature is established in Section~\ref{sec:positive-mean-curvature}.  The Hessian identities and pinching estimate are proved in Section~\ref{sec:classical-cone-identities}, and specific properties of the test function chosen are shown in Section~\ref{sec:properties-test-functions}, which together with the Hardy reduction, complete the proof of Theorem~\ref{thm:stable-homogeneous-cone-flatness}.

\section{Preliminaries}\label{sec:preliminaries}
In general we will work with $0<\gamma<1$, unless explicitly stated otherwise, and
\[
  \beta=\frac{2}{2-\gamma}, \qquad \apalpha=\beta\gamma = \frac{2\gamma}{2-\gamma}.
\]
In particular $0<\apalpha\le1$ when $0 < \gamma \le 2/3$. 

We start with the definition of   (classical) $\gamma$-Alt--Phillips solution (cf. \cite[Definition 2.2]{SavinYu2025StableMinimizingConesAltPhillips}) for $\gamma\in (0, 1)$. The same definition extends to lower regularity contexts in the viscosity formulation as in \cite{DeSilvaSavin2021DegenerateOnePhase}. It corresponds, in the smooth free boundary setting, to the criticality condition (or vanishing of the first inner variation) for the Alt--Phillips energy for $u \ge 0$:
\begin{equation}
  \label{eq:alt-phillips-energy} \mathcal{E}_\gamma(u, D) = \int_{D} \left(|\nabla u|^2 + u^\gamma \mathbbm{1}_{\{u > 0\}}\right) \, dx = \int_{\{u > 0\}\cap D} \left(|\nabla u|^2 + u^\gamma \right) \, dx.
\end{equation}
\begin{defn}\label{def:classical-homogeneous-ap-cone}
Let $n \ge 2$,  $D \subset \R^n$, and $\gamma \in (0, 1)$. We say that $u \ge 0$, $u\in C(D)$, is a \emph{classical solution to the $\gamma$-Alt--Phillips problem in $D$} if
\[
\left\{
    \begin{array}{rcll}
        \Delta u& = & \tfrac{\gamma}{2} u^{\gamma-1}&\quad\text{in}\quad \{u > 0\}\cap D,\\
        \tfrac{|\nabla u|^2}{u^\gamma} &=&1&\quad\text{on}\quad \partial\{u > 0\}\cap D,
    \end{array}
    \right.
\]
where $\{u > 0\}$ is locally around $\partial\{u >0\}$ the subgraph of a smooth function (i.e., $\{u > 0\}$ is a smooth domain).

We say that $u\in C(\R^n)$ is a \emph{classical cone} for the $\gamma$-Alt--Phillips problem if $D = \R^n$, $u$ is $\beta$-homogeneous, and $u$ is a classical solution to the $\gamma$-Alt--Phillips problem in $\R^n\setminus \{0\}$.
\end{defn}

We also define stable solutions (cf. \cite{KarakhanyanSanzPerela2026StableConesAltPhillips}):

\begin{defn}\label{def:stability}
Let $n \ge 2$,  $D \subset \R^n$, and $\gamma \in (0, 1)$. Let $u$ be a classical solution to the $\gamma$-Alt--Phillips problem in $D$. We say it is \emph{stable} if second-order inner variations are nonnegative; i.e.,
\[
  \lim_{\eps \downarrow 0} \frac{\mathcal{E}_\gamma(u_\eps, D)-\mathcal{E}_\gamma(u, D)}{\eps^2}\ge 0,\quad u_\eps(x) = u\circ ({\rm Id}+\eps \xi(x)), \quad\text{for all}\quad \xi \in C^\infty_c( D; \R^n).
\]
If $u$ is a classical cone for the $\gamma$-Alt--Phillips problem, we say it is stable if it satisfies the previous expression for variations compactly supported in $\R^n\setminus \{0\}$.
\end{defn}

\subsection{The transformed equation and stability}
We work in the transformed variable 
\[
v=\beta u^{1/\beta}.
\]

\begin{prop}[Transformed solution]
\label{prop:transformed-equation}
Let $n \ge 2$, $\gamma\in (0, 1)$. Let $u$ be a classical cone for the $\gamma$-Alt--Phillips problem in $\R^n$, and set
\[
  v=\beta u^{1/\beta}, \qquad u=\left(\frac v\beta\right)^\beta .
\]
Then, $v\in C(\R^n)$, $v \ge 0$ is 1-homogeneous, $\{u > 0 \} = \{v > 0\}$,  and satisfies
\[
\left\{
    \begin{array}{rcll}
        \Delta v& = & \frac{\apalpha}{2}\frac{1-|\grad v|^2}{v}&\quad\text{in}\quad \{v > 0\},\\
        |\nabla v|^2 &=& 1&\quad\text{on}\quad \partial\{v > 0\}.
    \end{array}
    \right.
\]
\end{prop}

\begin{proof}
This is the classical normalization for the problem and a computation yields it; see, for instance, \cite[Section 2.2]{KarakhanyanSanzPerela2026StableConesAltPhillips}.
\end{proof}

Of course, the transformation is not specific to conical solutions, and more in general, we have that the corresponding energy $\mathcal{E}_\gamma(u, D)$ in terms of $v$ is given by 
\begin{equation}
    \label{eq:venergy}
    {\mathcal{J}}_\gamma(v, D) =\int_{\{v > 0\}\cap D} v^\apalpha(|\nabla v|^2 +1). 
\end{equation}

We will work with stable solutions in the sense of inner variations as in Definition~\ref{def:stability}. By the recent work \cite{KarakhanyanSanzPerela2026StableConesAltPhillips}, we have the following stability inequality directly expressed in the transformed solution $v$.

\begin{prop}[Alt--Phillips stability inequality, {\protect \cite[Theorem 4.1]{KarakhanyanSanzPerela2026StableConesAltPhillips}}]
\label{prop:alt-phillips-stability-inequality}
Let $n \ge 2$, $\gamma\in (0, 1)$. Let $u$ be a classical stable cone for the $\gamma$-Alt--Phillips problem in $\R^n$, and let $v$ denote the transformed solution. Then,
\begin{equation}\label{eq:stability-inequality}
  \int_{\{v > 0\}} v^\apalpha
  \left(|\grad\varphi|^2
    -\frac{\apalpha}{2}\frac{1-|\grad v|^2}{v^2}\varphi^2
  \right)\,dx\ge0 \qquad\text{for all}\quad \varphi \in C^1_c(\R^n\setminus \{0\}).
\end{equation}
\end{prop}

It will be convenient to abbreviate
\begin{equation}\label{eq:pq-definitions}
  P=|\grad v|^2,
  \qquad
  Q=\frac{\apalpha}{2}\frac{1-P}{v^2}=\frac{\Delta v}{v}.
\end{equation}
With this notation, stability reads as
\[
  \int_{\{v > 0\}} v^\apalpha\bigl(|\grad\varphi|^2-Q\varphi^2\bigr)\,dx\ge0 \qquad\text{for all}\quad \varphi \in C^1_c(\R^n\setminus \{0\}).
\]
\subsection{Regularity estimates}
The following result is due to \cite[Theorem 1.4]{DeSilvaSavin2021DegenerateOnePhase}, where it is stated in the more general framework of viscosity solutions:

\begin{prop}[Lipschitz bound]
\label{prop:lipschitz-bound}
Let $n \ge 2$, $\gamma\in (0, 1)$. Let $u$ be a classical solution for the $\gamma$-Alt--Phillips problem in $B_1\subset \R^n$, and let $v$ denote the transformed solution. Assume $\partial\{v > 0\}\cap B_{1/2}\neq \varnothing$. Then,
\[
  |\grad v|\le C\quad\text{in}\quad B_{3/4},
\]
for some $C$ depending only on $n$ and $\gamma$.
\end{prop}

The following is a consequence of the higher order smoothness of the free boundary obtained in  \cite[Theorem 1.1 and Section 2]{RestrepoRosOton2025CinftyRegularitySemilinear}. We state it for cones, but it works for arbitrary classical solutions:
\begin{prop}[Higher regularity]
\label{prop:higher-regularity-boundary-calculus}
Let $n \ge 2$, $\gamma\in (0, 1)$. Let $u$ be a classical cone for the $\gamma$-Alt--Phillips problem in $\R^n$, and let $v$ denote the transformed solution. Let $\Omega = \{v > 0\}$, and let $\rho$ denote the (inward) distance to $\{v = 0\}$. Then, inside a fixed compact $K \Subset \R^n\setminus \{0\}$, we have the expansions around the free boundary:
\begin{equation}\label{eq:boundary-calculus}
  v=\rho+O(\rho^2),\qquad
  |\grad v|=1+O(\rho),\qquad
  D^2v=O(1),\qquad
  Q=O(\rho^{-1}).
\end{equation}
Consequently $v^\apalpha |Q|=O(\rho^{\apalpha-1})$ is locally integrable for $0<\apalpha<1$, and the nearby level sets $\{v=\eps\}\cap K$ have uniformly bounded area as $\eps\downarrow0$.
\end{prop}

\begin{remark}
\label{rem:lipschitz-admissible-tests}
As a consequence, in the stability inequality, Proposition~\ref{prop:alt-phillips-stability-inequality}, we may restrict ourselves to test functions $\varphi \in {\rm Lip}_c(\overline{\{v > 0\}}\setminus \{0 \})$. Indeed, from \eqref{eq:boundary-calculus}, $v^\apalpha |Q|=O(\rho^{\apalpha-1})\in L^1_{\rm loc}$, and  since $\varphi$ is Lipschitz, $\varphi\in H^1_{\rm loc}(v^\apalpha dx)$ as well. 

Then, mollifying $\varphi$ gives smooth compactly supported functions $\varphi_j$, converging uniformly to $\varphi$. Their gradients are uniformly bounded, and $\grad\varphi_j\to\grad\varphi$ a.e. up to subsequences.  Since $v^\apalpha$ is locally integrable, dominated convergence gives convergence in the weighted gradient norm.  The potential term also converges because
\[
  v^\apalpha |Q|\,|\varphi_j-\varphi|^2 \le C v^\apalpha |Q| =O(\rho^{\apalpha-1})\in L^1_{\rm loc}.
\]
Thus the stability inequality of Proposition~\ref{prop:alt-phillips-stability-inequality} passes to $\varphi$.
\end{remark}

\subsection{Notation}
Throughout the paper, we let $u$ be a classical cone for the $\gamma$-Alt--Phillips problem. We will denote by $v$ the corresponding transformed solution as in Proposition~\ref{prop:transformed-equation}.

The positive phase of a cone will often be written in polar coordinates as
\[
  \Om=\{r\omega:r>0,\,\omega\in \Om_S\}, \qquad \Om_S=\Om\cap\mathbb{S}^{n-1}, \qquad v(r,\omega)=r\psi(\omega).
\]
Throughout the paper, we denote
\[
  r=|x|, \qquad P=|\grad v|^2, \qquad Q=\frac{\apalpha}{2}\frac{1-P}{v^2}=\frac{\Delta v}{v}.
\]
We will also use the linearized operator defined below: it is obtained by linearizing
\[
  \Delta v-\frac{\apalpha}{2}\frac{1-|\nabla v|^2}{v}=0
\]
inside the positive phase.  If $v_\varepsilon=v+\varepsilon f$, then
\[
  \left.\frac{d}{d\varepsilon}\right|_{\varepsilon=0} \left( \Delta v_\varepsilon -\frac{\apalpha}{2}\frac{1-|\nabla v_\varepsilon|^2}{v_\varepsilon} \right) = \Delta f+\apalpha\frac{\nabla v}{v}\cdot\nabla f+Qf.
\]
The weighted Alt--Phillips linearized operator $\Lin = \Lin_{v, \apalpha}$ is then
\begin{equation}\label{eq:linearized-operator}
  \Lin f:=v^{-\apalpha}\operatorname{div}(v^\apalpha\grad f)+Qf
  =\Delta f+\apalpha\frac{\grad v}{v}\cdot\grad f+Qf.
\end{equation}
For homogeneous degree-zero scalar functions, we use the scale-invariant operator
\[
  \Jh f=r^2\Lin f.
\]

\begin{remark}[Boundary condition for Jacobi fields]
The operator $\Lin$ will be used below only as the interior weighted
linearized operator. If, however, $f$ arises from a smooth family of free-boundary solutions
$v_t$ (such as translations), $f=\partial_t v_t|_{t=0}$, then it also satisfies the linearized
Bernoulli condition at the boundary.  With $\nu=\nabla v/|\nabla v|$ the inward normal and
$H=\Delta\rho|_{\partial\Omega}$, one has $v_\nu  = 1$, $H = -(1+\apalpha)v_{\nu\nu}$ (see \cite{KarakhanyanSanzPerela2026StableConesAltPhillips}).
If the free boundary moves with normal speed $a$, differentiating
$v_t=0$ on the moving boundary gives $f+a=0$.  Differentiating
$|\nabla v_t|^2=1$ gives $f_\nu+a v_{\nu\nu}=0$.  Hence,
\[
  f_\nu-v_{\nu\nu}f = f_\nu+\frac{H}{1+\apalpha}f=0 \quad\text{on }\partial\Om.
\]
This boundary condition will not be imposed on the test functions used below;
the Hardy reduction uses interior subsolutions and the boundary contribution is
handled by the weighted cutoff near $\partial\Omega$.
\end{remark}
We will also define the Hardy constants
\[
\lambda_{\apalpha,p}^{(n)} := \frac{(n+\apalpha-2-2p)^2}{4}.
\]
 When fixing  $n = 3$, we will denote
\[
  \lambda_\apalpha = \lambda^{(3)}_{\apalpha, 0} =\frac{(1+\apalpha)^2}{4}\qquad \text{and}\qquad\lambda_{\apalpha,p}=\lambda^{(3)}_{\apalpha, p} = \frac{(1+\apalpha-2p)^2}{4}.
\]
 
The next definition is the main object of use at the end of the proof:

\begin{defn}[Homogeneous Hessian]
\label{def:homogeneous-hessian}
Let $n = 3$. The 0-homogeneous Hessian, defined in $\overline{\{v > 0\}}$, is
\[
  \HH:=rD^2v.
\]
In particular, by 1-homogeneity of $v$, $\HH x=0$.  On $x^\perp$, write
\[
  T=\tr (\HH) = \tr(\HH|_{x^\perp})=rQv, \qquad \lambda_M=\lamax(\HH|_{x^\perp}), \qquad \lambda_m=\lamin(\HH|_{x^\perp}).
\]
In polar coordinates $v=r\psi$,
\[
  \HH|_{x^\perp}=\grad_{\Sn}^2\psi+\psi g_{\Sn},
\]
where $g_{\Sn}$ is the standard round metric on $\Sn$. The Euclidean largest Hessian eigenvalue used later is
\[
  \Lambda=\lamax(D^2v)=\frac{\lambda_M}{r}.
\]
\end{defn}
The flat half-space solution is characterized by $\HH\equiv0$ in its positive phase.

\section{A reduction on stability}\label{sec:reduction}
The stability inequality will be used through the following reduction (cf. \cite[Section 2.3, Propositions 2.1--2.2]{JerisonSavin2015StableConesOnePhase}).

\begin{prop}
\label{prop:stability-hardy-reduction}
Let $n \ge 2$, $\gamma\in (0, 1)$. Let $u$ be a classical stable cone for the $\gamma$-Alt--Phillips problem in $\R^n$, and let $v$ denote the transformed solution.

Let $G\ge0$, $G\not\equiv0$ in $\Om=\{v>0\}\subset \R^n$, be homogeneous of degree $-p$ in $\Om = \{v > 0\}$. Assume that $G\in {\rm Lip}_{\rm loc}(\overline\Om\setminus\{0\})$ and 
\begin{equation}\label{eq:g-subsolution-distributional}
  \Lin G\ge\mu\frac{G}{r^2}
\end{equation}
holds in the distributional sense in $\Om$ (recall \eqref{eq:linearized-operator}).  Then
\[
  \mu\le\lambda_{\apalpha,p}^{(n)} := \frac{(n+\apalpha-2-2p)^2}{4}.
\]
\end{prop}

\begin{proof}
We separate the proof into two steps.

\medskip
\noindent {\bf Step 1:} The distributional assumption means that, for every nonnegative $\zeta\in C_c^\infty(\Om)$,
\begin{equation}\label{eq:weighted-distributional-subsolution}
        -\int_\Om v^\apalpha \grad G\cdot\grad\zeta\,dx
        +\int_\Om v^\apalpha QG\zeta\,dx
        \ge
        \mu\int_\Om v^\apalpha\frac{G\zeta}{r^2}\,dx .
\end{equation}

Fix $\eta\in C_c^\infty((0,\infty))$, and let $K\Subset\R^n\setminus\{0\}$ be an annulus containing the support of $\eta(r)$.  We first justify that \eqref{eq:weighted-distributional-subsolution} may be tested against $\eta^2G$.

Choose $\chi\in C^\infty(\R)$ with
$   0\le\chi\le1$, $\chi=0$ on  $(-\infty,1]$, $\chi=1$ on $[2,\infty)$ 
and set
$  \chi_\eps(x):=\chi\!\left(\frac{v(x)}{\eps}\right)$.
For fixed $\eps>0$, the function
\[
  \zeta_\eps:=\chi_\eps\eta^2G
\]
is nonnegative, Lipschitz, and compactly supported in $\Om\setminus\{0\}$. Hence, it is admissible, after a standard mollification argument, in \eqref{eq:weighted-distributional-subsolution}:
\[
\begin{split}
        &-\int_\Om v^\apalpha \grad G\cdot\grad(\chi_\eps\eta^2G)\,dx
        +\int_\Om v^\apalpha Q\chi_\eps\eta^2G^2\,dx     \ge
        \mu\int_\Om v^\apalpha\frac{\chi_\eps\eta^2G^2}{r^2}\,dx .
\end{split}
\]
We let $\eps\downarrow0$.  The terms not containing $\grad\chi_\eps$ converge by dominated convergence, using that $G$ is Lipschitz on $K$, that $r$ is bounded above and below on $K$, and that
$
  v^\apalpha |Q|\in L^1_{\rm loc}(\overline\Om\setminus\{0\})
$
by Proposition~\ref{prop:higher-regularity-boundary-calculus}.  It remains to check the cutoff error
\[
  E_\eps := \int_\Om v^\apalpha\eta^2G\,\grad G\cdot\grad\chi_\eps\,dx .
\]
Since
$
  |\grad\chi_\eps| \le \frac{C}{\eps}|\grad v|, $ $\supp\grad\chi_\eps\subset\{\eps<v<2\eps\}$,
and since $G,\grad G,\eta$ are bounded on $K$, the coarea formula gives
\[
\begin{split}
        |E_\eps|
        &\le
        \frac{C_K}{\eps}
        \int_{\{\eps<v<2\eps\}\cap K}
        v^\apalpha |\grad v|\,dx        = \frac{C_K}{\eps}
        \int_\eps^{2\eps}
        t^\apalpha
        \mathcal H^{n-1}(\{v=t\}\cap K)\,dt .
\end{split}
\]
The level-set measures are uniformly bounded for small $t>0$.  Therefore
\[
  |E_\eps|\le C_K\eps^\apalpha\to0 .
\]
Thus \eqref{eq:weighted-distributional-subsolution} can be tested against $\eta^2G$:
\begin{equation}\label{eq:test-against-eta-g}
        -\int_\Om v^\apalpha \grad G\cdot\grad(\eta^2G)\,dx
        +\int_\Om v^\apalpha Q\eta^2G^2\,dx
        \ge
        \mu\int_\Om v^\apalpha\frac{\eta^2G^2}{r^2}\,dx .
\end{equation}

\medskip
\noindent {\bf Step 2:} 
We now use stability (cf.  \cite{CaffarelliJerisonKenig2004GlobalMinimizersFreeBoundary}).  Since $\eta G$ is admissible (by Remark~\ref{rem:lipschitz-admissible-tests}),
\[
\begin{split}
        0 &\le
        \int_\Om v^\apalpha
        \bigl(|\grad(\eta G)|^2-Q\eta^2G^2\bigr)\,        = \int_\Om v^\apalpha|\eta'|^2G^2\, 
        + \int_\Om v^\apalpha \grad G\cdot\grad(\eta^2G)\, 
        - \int_\Om v^\apalpha Q\eta^2G^2   ,
\end{split}
\]
which is the distributional form of the identity
\begin{equation}
    \label{eq:identity-dist-form}
0\le \int_\Omega v^\alpha\bigl(|\nabla(\eta G)|^2-Q\eta^2G^2\bigr)
=
\int_\Omega v^\alpha|\eta'|^2G^2
-\int_\Omega v^\alpha\eta^2G\,\Lin G.
\end{equation}
Combined with \eqref{eq:test-against-eta-g}, we obtain
\begin{equation}\label{eq:before-radial-hardy}
        \mu\int_\Om v^\apalpha\frac{\eta^2G^2}{r^2}\,dx
        \le
        \int_\Om v^\apalpha|\eta'|^2G^2\,dx .
\end{equation}

Finally we reduce to the radial Hardy quotient.  Write
\[
  v(r,\omega)=r\psi(\omega), \qquad G(r,\omega)=r^{-p}F(\omega), \qquad \Omega_S=\Om\cap\mathbb S^{n-1}.
\]
Set
\[
  A_G:=\int_{\Omega_S}\psi^\apalpha F^2\,d\omega .
\]
Since $G\ge0$ and $G\not\equiv0$ in $\Om$, we have $A_G>0$.  By homogeneity and spherical coordinates,
\[
  \int_\Om v^\apalpha|\eta'|^2G^2\,dx = A_G\int_0^\infty r^{n-1+\apalpha-2p}|\eta'|^2\,dr ,
\]
and
\[
  \int_\Om v^\apalpha\frac{\eta^2G^2}{r^2}\,dx = A_G\int_0^\infty r^{n-3+\apalpha-2p}\eta^2\,dr .
\]
Thus, \eqref{eq:before-radial-hardy} gives
\[
  \mu \int_0^\infty r^{n-3+\apalpha-2p}\eta^2\,dr \le \int_0^\infty r^{n-1+\apalpha-2p}|\eta'|^2\,dr
\]
for every $\eta\in C_c^\infty((0,\infty))$.  Taking the infimum over nonzero such $\eta$ and using the sharp one-dimensional Hardy constant,
\[
\begin{split}
        \mu
        &\le
        \inf_{\eta\not\equiv0}
        \frac{\displaystyle
        \int_0^\infty r^{n-1+\apalpha-2p}|\eta'|^2\,dr}
        {\displaystyle
        \int_0^\infty r^{n-3+\apalpha-2p}\eta^2\,dr}  =
        \frac{(n+\apalpha-2-2p)^2}{4}
        = \lambda_{\apalpha,p}^{(n)} .
\end{split}
\]
The proposition follows.
\end{proof}

A consequence from the homogeneous formulation is the following.  It excludes cones whose positivity set is all of $\R^n\setminus\{0\}$ for certain ranges of $\gamma$. The specific ranges match those obtained in \cite{SavinYu2025StableMinimizingConesAltPhillips} for stable cones with radial symmetry.

\begin{prop} \label{prop:full-support-exclusion}
Let $n \ge 3$, $\gamma\in (0, 1)$, $ \apalpha=\frac{2\gamma}{2-\gamma} $. Let $u$ be a classical stable cone for the $\gamma$-Alt--Phillips problem in $\R^n$, and let $v$ denote the transformed solution.

Assume that $\Omega := \{v > 0\} = \R^n\setminus \{0\}$. Then
\[
  n-1\le \frac{(n+\apalpha-2)^2}{4}.
\]
In particular, if $n \le 6$,  then there is no full-support stable cone when
\[
  0<\gamma< \gamma_n := \frac{2(2\sqrt{n-1}-n+2)} {2\sqrt{n-1}-n+4}.
\]
Moreover, if $\gamma = \gamma_n$, then the cone is radial.
\end{prop}

\begin{proof}
Write
\[
  v(r,\omega)=r\psi(\omega), \qquad \psi>0 \quad\text{on } \mathbb S^{n-1}.
\]
Then
\[
  Q=r^{-2}q(\omega), \qquad q=\frac{\Delta_{\mathbb S^{n-1}}\psi+(n-1)\psi}{\psi}.
\]
For degree-zero functions $F=F(\omega)$,
\[
  r^2\Lin F = \Delta_{\mathbb S^{n-1}}F + \apalpha \frac{\grad_{\mathbb S^{n-1}}\psi}{\psi} \cdot \grad_{\mathbb S^{n-1}}F + qF .
\]
Take
\[
  G:=\frac1\psi=\frac r v .
\]
A direct computation gives
\[
  \grad_{\mathbb S^{n-1}}G = -\psi^{-2}\grad_{\mathbb S^{n-1}}\psi
\]
and
\[
  \Delta_{\mathbb S^{n-1}}G = -\psi^{-2}\Delta_{\mathbb S^{n-1}}\psi + 2\psi^{-3}|\grad_{\mathbb S^{n-1}}\psi|^2 .
\]
Therefore
\[
\begin{split}
        r^2\Lin G
        &= \Delta_{\mathbb S^{n-1}}G
        + \apalpha
        \frac{\grad_{\mathbb S^{n-1}}\psi}{\psi}
        \cdot
        \grad_{\mathbb S^{n-1}}G
        + qG                                                    \\
        &= -\psi^{-2}\Delta_{\mathbb S^{n-1}}\psi
        +(2-\apalpha)\psi^{-3}
        |\grad_{\mathbb S^{n-1}}\psi|^2
        + \frac{\Delta_{\mathbb S^{n-1}}\psi+(n-1)\psi}{\psi^2} \\
        &= (n-1)G
        + (2-\apalpha)
        \frac{|\grad_{\mathbb S^{n-1}}\psi|^2}{\psi^2}G .
\end{split}
\]
Since $0<\gamma<1$, we have $0<\apalpha<2$.  Hence
\[
  \Lin G\ge (n-1)\frac{G}{r^2}.
\]
Applying Proposition~\ref{prop:stability-hardy-reduction} with $p=0$ and $\mu=n-1$ gives
\[
  n-1\le \lambda_{\apalpha,0}^{(n)} = \frac{(n+\apalpha-2)^2}{4}.
\]
This is equivalent to
\[
  \apalpha\ge 2\sqrt{n-1}-n+2.
\]
The corresponding threshold in $\gamma$ is
\[
  \gamma_n= \frac{2(2\sqrt{n-1}-n+2)} {2\sqrt{n-1}-n+4}
\]
whenever $2\sqrt{n-1}-n+2>0$ (i.e., $n \le 6$).

It remains to deal with the equality case.  Keeping the positive remainder in the stability computation \eqref{eq:identity-dist-form} with the test function
\[
  \varphi(r,\omega)=\eta(r)\psi(\omega)^{-1}
\]
gives
\[
  0\le A\int_0^\infty r^{n-1+\apalpha}|\eta'|^2\,dr - \bigl((n-1)A+(2-\apalpha)E\bigr) \int_0^\infty r^{n-3+\apalpha}\eta^2\,dr,
\]
where
\[
  A:=\int_{\mathbb S^{n-1}}\psi^{\apalpha-2}\,d\omega>0, \qquad E:= \int_{\mathbb S^{n-1}} \psi^{\apalpha-4} |\grad_{\mathbb S^{n-1}}\psi|^2\,d\omega\ge0 .
\]
Taking a sharp Hardy sequence yields
\[
  (n-1)A+(2-\apalpha)E \le \frac{(n+\apalpha-2)^2}{4}A .
\]
If equality holds in the threshold, then
\[
  \frac{(n+\apalpha-2)^2}{4}=n-1.
\]
Since $0<\apalpha<2$, this forces
$
  E=0.
$
Thus
\[
  \grad_{\mathbb S^{n-1}}\psi\equiv0,
\]
so $\psi$ is constant and the cone is radial.
\end{proof}

\section{Global positive mean curvature}\label{sec:positive-mean-curvature}
In this section we show that the mean curvature of the free boundary is positive. This matches the corresponding property  known for $\gamma= 0$, and it may be already known in the literature. We record it here for completeness.

The starting point is the sharp transformed gradient bound $|\grad v|\le1$ \'a la Modica (see Remark~\ref{rem:positive-mean-curvature-cones} below for the same results in classical cones):

\begin{prop}[Global gradient bound]
\label{prop:global-gradient-bound}
Let $n \ge 2$, $\gamma\in (0, 1)$. Let $u$ be a global classical solution for the $\gamma$-Alt--Phillips problem, and let $v$ denote the transformed solution. Assume $\partial\{v > 0\}\neq \varnothing$.  Then
\[
  |\grad v|\le1\quad\text{in }\Om .
\]
\end{prop}

\begin{proof}
Let $L:=\sup_{\R^n}|\grad v|$, which by Proposition~\ref{prop:lipschitz-bound} (rescaled) satisfies $L  < \infty$.   If $L \le 1$ there is nothing to prove.  Suppose instead that $L>1$.  Choose points $x_k\in\Om$ such that $|\grad v(x_k)|\to L$, and set $\lambda_k=v(x_k)>0$.  Define
\[
  v_k(y)=\frac{v(x_k+\lambda_k y)}{\lambda_k}.
\]
Since $v$ is globally $L$-Lipschitz and vanishes continuously on $\partial\Om$,
\[
  v(x_k)\le L\,\dist(x_k,\partial\Om),
\]
and hence $B_{\lambda_k/(2L)}(x_k)\subset\Om$.  After scaling, $v_k$ is a smooth solution of the same transformed equation in $B_{1/(2L)}$, and
\[
  v_k(0)=1, \qquad |\grad v_k(0)|\to L, \qquad |\grad v_k|\le L .
\]
Moreover $v_k(y)\ge1-L|y|$, so $v_k\ge1/2$ in $B_{1/(2L)}$.  In this fixed ball the right-hand side of the equation is smooth and uniformly controlled.  Interior estimates and Arzel\`a--Ascoli therefore give, after passing to a subsequence, smooth convergence in a smaller ball to a solution $V$ satisfying
\[
  V(0)=1, \qquad |\grad V(0)|=L, \qquad |\grad V|\le L,
\]
and
\[
  \Delta V=\frac{\apalpha}{2}\frac{1-|\grad V|^2}{V}.
\]
Set $P=|\grad V|^2$.  Differentiating the equation gives
\begin{equation}\label{eq:modica-bochner-limit}
  \left(\Delta+\apalpha\frac{\grad V}{V}\cdot\grad\right)P
  =2|D^2V|^2+\apalpha\frac{P(P-1)}{V^2} .
\end{equation}
Since $P\le L^2$ and $P(0)=L^2$, the function $P$ has an interior maximum at the origin.  Thus $\grad P(0)=0$ and $\Delta P(0)\le0$.  Evaluating \eqref{eq:modica-bochner-limit} at $0$ gives
\[
  0\ge 2|D^2V(0)|^2+ \apalpha L^2\frac{L^2-1}{V(0)^2}>0,
\]
a contradiction.  Therefore $L\le1$, as claimed.
\end{proof}

We now convert the sharp gradient estimate into strict positivity of the inward mean curvature.  The proof is local near a regular free-boundary patch, but it uses the global sign $1-|\grad v|^2\ge0$ (in particular, by Remark~\ref{rem:positive-mean-curvature-cones} it applies to classical cones as well).

\begin{prop}[Positive mean curvature]
\label{prop:positive-mean-curvature}
Let $n \ge 2$, $\gamma\in (0, 1)$. Let $u$ be a global classical solution for the $\gamma$-Alt--Phillips problem, and let $v$ denote the transformed solution. Let $\mathcal{V}$ denote a nonflat connected component of $\{v > 0\}$.

Let $H:\partial\{v > 0\}\to \R$ denote the mean curvature of $\partial \{v > 0\}$ with respect to the inner unit normal vector to $\{v > 0\}$. Then,
\[
 |\nabla v|^2  < 1 \quad\text{in}\quad\mathcal{V}\qquad\text{and}\qquad  H > 0\quad\text{on}\quad \partial\mathcal{V}.
\]
\end{prop}

\begin{proof}
Let $\rho$ be the inward distance to $\Sigma :=\partial\mathcal{V}$, so that $ H:=\Delta\rho\big|_\Sigma $ on $\partial\mathcal{V}$.

Let
\[
  S:=1-|\grad v|^2 .
\]
By Proposition~\ref{prop:global-gradient-bound}, $S\ge0$ in $\{v > 0\}$, and $S=0$ on $\Sigma$.  We first show that, unless $v$ is flat in $\mathcal{V}$,
\begin{equation}\label{eq:s-positive-interior}
  S>0\quad\text{in }\mathcal{V}.
\end{equation}
Indeed, recalling $P=|\grad v|^2$ and using $\Delta v=\frac{\apalpha}{2}S/v$, we have
\begin{equation}\label{eq:s-equation}
  \Delta S
  +\apalpha\frac{\grad v}{v}\cdot\grad S
  -\apalpha\frac{P}{v^2}S
  =-2|D^2v|^2\le0 .
\end{equation}
The zero-order coefficient is nonpositive.  If $S$ had an interior zero, the strong maximum principle would yield $S\equiv0$ on the component.  Then $\Delta v=0$  and $|\grad v|^2\equiv1$ there, so  $D^2v\equiv0$.  This is the flat case. Hence, $ S > 0$ in $\mathcal{V}$.

Fix $x_0\in\Sigma$.  Let $ \nu=\grad v $ be the inward unit normal to $\mathcal{V}$. With the convention $H=\Delta\rho|_\Sigma$, the boundary identity in \cite[Lemma 5.1]{KarakhanyanSanzPerela2026StableConesAltPhillips} gives
\begin{equation}\label{eq:h-vrr-boundary-identity}
  H=-(1+\apalpha)v_{\nu\nu}
  \quad\text{at }x_0 .
\end{equation}
Moreover, all tangential derivatives of $v$ vanish on $\Sigma$, and $v_\nu=1$.  Therefore
\begin{equation}\label{eq:s-rho-h-boundary}
  \partial_\nu S(x_0)
  =-\partial_\nu |\grad v|^2(x_0)
  =-2v_{\nu\nu}(x_0)
  =\frac{2H(x_0)}{1+\apalpha} .
\end{equation}

We thus have to prove $\partial_\nu S(x_0)>0$.

Define
\[
  \mathcal M f := \Delta f +\apalpha\frac{\grad v}{v}\cdot\grad f -\apalpha\frac{|\grad v|^2}{v^2}f .
\]
Then $\mathcal M S\le0$ by \eqref{eq:s-equation}.  Work in Fermi coordinates $(y,\rho)$ centered at $x_0$, with $y=0$ corresponding to $x_0$.  By Proposition~\ref{prop:higher-regularity-boundary-calculus}, after shrinking the coordinate cylinder if necessary,
\[
  \frac{\partial_\rho v}{v}=\frac1\rho+O(1), \qquad \frac{|\grad v|^2}{v^2}=\frac1{\rho^2}+O\!\left(\frac1\rho\right), \qquad \Delta\rho=O(1).
\]
We have denoted here by $\partial_\rho$ the derivative along the coordinate vector field in the inward normal direction in Fermi coordinates near the free boundary. Also, since $v(y,0)=0$ and $v_\rho(y,0)=1$ along the free boundary, $\grad_y v=O(\rho^2)$.

Fix $R>0$ to work in a local coordinate patch, and for  constants $B,C>0$, set
\[
  \Phi(y,\rho)=\rho+C\rho^2-B|y|^2
\]
in
\[
  \mathcal C_{R,\eps}:=\{|y|<R,\ 0<\rho<\eps\}.
\]
For $\Phi_0(\rho)=\rho+C\rho^2$, the singular normal terms cancel:
\[
  \apalpha\frac{\partial_\rho v}{v}\Phi_0' -\apalpha\frac{|\grad v|^2}{v^2}\Phi_0 =\apalpha C+O(1),
\]
while (using that $\Delta f(\rho) = f''(\rho) +f'(\rho)\Delta\rho$):
\[
  \Delta\Phi_0=2C+O(1).
\]
Thus
\begin{equation}\label{eq:m-phi-main-term}
  \mathcal M\Phi_0=(2+\apalpha)C+O(1).
\end{equation}
The tangential Laplacian and drift terms acting on $-B|y|^2$ are bounded below by $-C_0B$ in a fixed small cylinder, while the zero-order term is nonnegative:
\[
  -\apalpha\frac{|\grad v|^2}{v^2}(-B|y|^2) =\apalpha B\frac{|\grad v|^2}{v^2}|y|^2\ge0 .
\]
Consequently
\begin{equation}\label{eq:m-tangential-penalty}
  \mathcal M(-B|y|^2)\ge -C_0B .
\end{equation}
Choose $B>0$, then choose $C$ large enough so that the main term in \eqref{eq:m-phi-main-term} dominates \eqref{eq:m-tangential-penalty} and the bounded errors.  Shrinking $\eps$ if needed, we get
\begin{equation}\label{eq:m-phi-positive}
  \mathcal M\Phi\ge1
  \quad\text{in }\mathcal C_{R,\eps}.
\end{equation}
We also take $\eps$ small enough that
\begin{equation}\label{eq:phi-lateral-negative}
  \eps+C\eps^2\le BR^2 .
\end{equation}
Then $\Phi\le0$ on the lateral face $|y|=R$, and $\Phi(y,0)=-B|y|^2\le0$ on the free-boundary face.

By \eqref{eq:s-positive-interior},
\[
  S>0 \quad\text{on}\quad {\{|y|\le R,\ \rho=\eps\}}.
\]
Since $\Phi$ is bounded above on the top face, for sufficiently small $\delta>0$ we have $w:= S-\delta\Phi\ge0$ there as well, so $w\ge0$ holds in the whole $\partial\mathcal C_{R,\eps}$, while
\[
  \mathcal M w =\mathcal M S-\delta\mathcal M\Phi \le -\delta<0 \quad\text{in }\mathcal C_{R,\eps}.
\]
By maximum principle as before (the zero-order term in $\mathcal M$ is nonpositive), we get $w\ge0$ in $ \mathcal C_{R,\eps}$. Along the inward normal line $y=0$, $\Phi(0,\rho)=\rho+C\rho^2$.  Hence
\[
  \liminf_{\rho\downarrow0}\frac{S(0,\rho)}{\rho}\ge\delta .
\]
Since $S$ is smooth up to the boundary and $S(x_0)=0$, this gives $\partial_\nu S(x_0)\ge\delta>0$, as we wanted.
\end{proof}
\begin{remark}[The conical case]
\label{rem:positive-mean-curvature-cones}
The conclusion of Propositions~\ref{prop:global-gradient-bound} and \ref{prop:positive-mean-curvature} will also be used for
classical cones. Notice that $|\nabla v|< 1$ by a directly analogous argument in this case: If, for $P = |\nabla v|^2$, $\sup_{\mathcal{V}} P>1$, then the restriction of $P$ to  
$\mathcal{V}_S$ attains a maximum $>1$ at an interior point, because $P=1$ on
the free boundary. At such a point, using \eqref{eq:s-equation} we get the contradiction. Thus $|\nabla v|\le 1$ in $V$.
\end{remark}

\section{Identities for classical cones}\label{sec:classical-cone-identities}
In this section, we will deal with classical cones with $\partial\Om\setminus\{0\}\ne\varnothing$, and prove various useful identities. This will be assumed throughout; otherwise, we refer to Proposition~\ref{prop:full-support-exclusion}.

The operator $\Jh$ will always act on homogeneous functions. We recall $\HH$ is given in Definition~\ref{def:homogeneous-hessian} and define its Euclidean components as
\[
  \mathcal A_{ij}(x):=|x|v_{ij}(x),
\]
which are 0-homogeneous. We remark that all derivatives presented here are Euclidean, unless otherwise stated.

\begin{lem} 
\label{lem:tensor-identities}
Let $n = 3$, $\gamma\in (0, 1)$. Let $u$ be a  classical cone for the $\gamma$-Alt--Phillips problem in $\R^3$, and let $v$ denote the transformed solution.

For every Euclidean unit vector $\tau\perp x$,
\begin{equation}\label{eq:p-and-t-derivatives}
  r\,\partial_\tau P=2\HH(\grad v,\tau),
  \qquad
  \partial_\tau T=-\frac1v
  \bigl(\apalpha\HH(\grad v,\tau)+T\,\partial_\tau v\bigr).
\end{equation}

Fix $e_1\in\Om$ (after a rotation) and put $I=\{2,3\}$.  All quantities below are evaluated at $e_1$.  For $i,j\in I$, set
\[
  \mathcal A_{ij}(x)=|x|v_{ij}(x), \qquad A_{ij}=\mathcal A_{ij}(e_1), \qquad T=A_{22}+A_{33}.
\]
Then
\[
  v_1=v, \qquad A_{1j}=0, \qquad T=\Delta v.
\]
and, for $\ell\in I$,
\[
  P_\ell=2\sum_{k\in I}A_{\ell k}v_k, \qquad\text{and}\qquad T_\ell =-\frac1v \left( \apalpha\sum_{k\in I}A_{\ell k}v_k+Tv_\ell \right).
\]
Finally, for $i,j\in I$, the scalar functions $\mathcal A_{ij}$ satisfy
\begin{equation}\label{eq:scalar-component-hessian-identity}
\begin{split}
  \Jh(\mathcal A_{ij})
  &=\apalpha A_{ij}
    -\frac{\apalpha}{v}\sum_{k\in I}A_{ik}A_{kj}
    +\frac{\apalpha}{v^2}
        \sum_{k\in I}(A_{ik}v_j         +A_{jk}v_i)v_k
    +\frac{2T}{v^2}v_iv_j .
\end{split}
\end{equation}
\end{lem}

In particular, if the basis $e_2,e_3$ diagonalizes $\HH|_{e_1^\perp}$, say
\[
  A_{22}=\lambda_M, \qquad A_{33}=\lambda_m, \qquad A_{23}=0,
\]
then
\[
  T_2=-\frac{T+\apalpha \lambda_M}{v}v_2, \qquad T_3=-\frac{T+\apalpha \lambda_m}{v}v_3,
\]
\[
  \Jh(\mathcal A_{22}) =\apalpha \lambda_M-\frac{\apalpha}{v}\lambda_M^2 +\frac{2(T+\apalpha \lambda_M)}{v^2}v_2^2,
\]
\[
  \Jh(\mathcal A_{33}) =\apalpha \lambda_m-\frac{\apalpha}{v}\lambda_m^2 +\frac{2(T+\apalpha \lambda_m)}{v^2}v_3^2,
\]
and
\[
  \Jh(\mathcal A_{23}) =\frac{(2+\apalpha)T}{v^2}v_2v_3 .
\]
\begin{proof}[Proof of Lemma~\ref{lem:tensor-identities}]
The first identities \eqref{eq:p-and-t-derivatives} are immediate.

We now work at $e_1$.  Homogeneity gives
\[
  v_1=v, \qquad v_{1j}=0,
\]
and therefore $A_{1j}=0$.  Since $r=1$ at the point, $T=A_{22}+A_{33}=\Delta v$.

For $\ell\in I$,
\[
  P_\ell=2\sum_{m=1}^3v_mv_{m\ell} =2\sum_{k\in I}v_kv_{k\ell} =2\sum_{k\in I}A_{\ell k}v_k .
\]
Differentiating
\[
  T=\frac{\apalpha r}{2}\frac{1-P}{v}
\]
in a tangential direction, and using $r_\ell=0$ at $e_1$, gives
\[
  T_\ell =-\frac{\apalpha}{2v}P_\ell-\frac{T}{v}v_\ell =-\frac1v \left( \apalpha\sum_{k\in I}A_{\ell k}v_k+Tv_\ell \right).
\]
It remains to compute $\Jh(\mathcal A_{ij})$.  Since $v_{ij}$ is homogeneous of degree $-1$,
\[
  v_{ij1}=-v_{ij} \qquad\text{at }e_1.
\]
Thus
\[
  (\mathcal A_{ij})_1=v_{ij}+v_{ij1}=0.
\]
Moreover,
\[
  \Delta(rv_{ij}) =r\Delta v_{ij}+2\grad r\cdot\grad v_{ij}+v_{ij}\Delta r.
\]
At $e_1$, $\grad r=e_1$, $\Delta r=2$, and $v_{ij1}=-v_{ij}$, so the last two terms cancel and
\[
  \Delta\mathcal A_{ij}=\Delta v_{ij}=(\Delta v)_{ij}.
\]
For $k\in I$, $(\mathcal A_{ij})_k(e_1)=v_{ijk}$.  Therefore, with $F=\Delta v$,
\begin{equation}\label{eq:jh-component-before-fij}
  \Jh(\mathcal A_{ij})(e_1)
  =F_{ij}
   +\frac{\apalpha}{v}\sum_{k\in I}v_kv_{ijk}
   +\frac{T}{v}A_{ij} .
\end{equation}

Now $ F=\frac{\apalpha}{2}\frac{1-P}{v}, $ for $i\in I$, at $e_1$,
\[
  F_i =-\frac{\apalpha}{v}M_i-\frac{T}{v}v_i, \qquad\text{where}\quad M_i:=\sum_{k\in I}A_{ik}v_k.
\]
Differentiating once more gives
\[
  F_{ij} =-\frac{\apalpha}{v}\partial_j \left(\sum_{m=1}^3v_mv_{mi}\right) +\frac{\apalpha}{v^2}M_iv_j -\frac{T_j}{v}v_i -\frac{T}{v}A_{ij} +\frac{T}{v^2}v_iv_j .
\]
The derivative in the first term is
\[
\begin{split}
  \partial_j\left(\sum_{m=1}^3v_mv_{mi}\right)
  &=\sum_{m=1}^3v_{mj}v_{mi}
    +\sum_{m=1}^3v_mv_{mij}   =\sum_{k\in I}A_{kj}A_{ki}-vA_{ij}
    +\sum_{k\in I}v_kv_{ijk}.
\end{split}
\]
Substituting this into the expression for $F_{ij}$, and then into \eqref{eq:jh-component-before-fij}, cancels the terms involving $\sum v_kv_{ijk}$ and $ T A_{ij}/v$, $T A_{ij}/v$, so we get
\[
  \Jh(\mathcal A_{ij}) =\apalpha A_{ij} -\frac{\apalpha}{v}\sum_{k\in I}A_{ik}A_{kj} +\frac{\apalpha}{v^2}M_iv_j -\frac{T_j}{v}v_i +\frac{T}{v^2}v_iv_j .
\]
Using the already proved trace-gradient identity,
\[
  T_j=-\frac1v(\apalpha M_j+Tv_j),
\]
we obtain \eqref{eq:scalar-component-hessian-identity}.
\end{proof}

We next account for the motion of the tangent frame.

\begin{lem} 
\label{lem:moving-plane-corrections}
In the setting of  Lemma~\ref{lem:tensor-identities} we have
\begin{equation}\label{eq:first-derivative-trace-correction}
        T_i=(\mathcal A_{22})_i+(\mathcal A_{33})_i,
        \qquad i=2,3.
\end{equation}
Moreover,
\begin{equation}\label{eq:moving-trace-correction}
        \Jh T
        = \Jh(\mathcal A_{22})+\Jh(\mathcal A_{33})+2T.
\end{equation}
If $\lambda_M>\lambda_m$, then the largest eigenvalue
$
  \lambda_{\max}(\HH|_{x^\perp})
$
is smooth near $e_1$.  Denoting this local eigenvalue function again by $\lambda_M$, one has
\[
  (\lambda_M)_i=(\mathcal A_{22})_i, \qquad i=2,3,
\]
and
\begin{equation}\label{eq:moving-eigenvalue-correction-max}
        \Jh \lambda_M
        = \Jh(\mathcal A_{22})+2\lambda_M
        + \frac{2}{\lambda_M-\lambda_m}
        \left((\mathcal A_{22})_3^2+(\mathcal A_{33})_2^2\right),
\end{equation}
\begin{equation}\label{eq:moving-eigenvalue-correction-min}
        \Jh \lambda_m
        = \Jh(\mathcal A_{33})+2\lambda_m
        - \frac{2}{\lambda_M-\lambda_m}
        \left((\mathcal A_{22})_3^2+(\mathcal A_{33})_2^2\right).
\end{equation}
\end{lem}

\begin{remark}
    In the context of the present manuscript, the expression \eqref{eq:moving-eigenvalue-correction-min} will not be used. We have still recorded it here for the reader's convenience. 
\end{remark}
\begin{proof}
All quantities are evaluated at $e_1$.  Since $v$ is one-homogeneous,
\[
  \HH(x)x=0.
\]
Hence
\[
  T= \tr(\HH|_{x^\perp})=\tr\HH = \mathcal A_{11}+\mathcal A_{22}+\mathcal A_{33}.
\]
We compute the contribution of $\mathcal A_{11}$ on the Euclidean derivatives, which we want to show is zero.  Fix $i\in\{2,3\}$, and set
\[
  y_i(s):=e_1+s e_i.
\]
Since $\HH(y_i(s))y_i(s)=0$, we have
\begin{equation}
\label{eq:zero-homogeneous-identities}
\begin{split}
        \mathcal A_{1i}(y_i(s))
        +         s\mathcal A_{ii}(y_i(s))
        & =        0,\\
        \mathcal A_{11}(y_i(s))
        +        s\mathcal A_{1i}(y_i(s))
        & =
        0.
        \end{split}
\end{equation}
Therefore
\[
  \mathcal A_{11}(y_i(s)) = s^2\mathcal A_{ii}(y_i(s)).
\]
Differentiating at $s=0$ gives
\[
  (\mathcal A_{11})_i=0, \qquad (\mathcal A_{11})_{ii}=2A_{ii}, \qquad i=2,3.
\]
Consequently,
\[
  T_i = (\mathcal A_{11})_i+(\mathcal A_{22})_i+(\mathcal A_{33})_i = (\mathcal A_{22})_i+(\mathcal A_{33})_i,
\]
which proves \eqref{eq:first-derivative-trace-correction}.

We next prove the trace equality.  For any homogeneous degree-zero scalar function $f$,
\[
  \Jh f = f_{22}+f_{33} + \apalpha\frac{v_2f_2+v_3f_3}{v} + \frac{T}{v}f .
\]
Apply this to $f=\mathcal A_{11}$.  Since, at $e_1$,
\[
  \mathcal A_{11}=0, \qquad (\mathcal A_{11})_2=(\mathcal A_{11})_3=0,
\]
the drift and zero-order terms vanish, and hence
\[
  \Jh(\mathcal A_{11}) = (\mathcal A_{11})_{22}+(\mathcal A_{11})_{33} = 2A_{22}+2A_{33} = 2T.
\]
Using
\[
  T=\mathcal A_{11}+\mathcal A_{22}+\mathcal A_{33},
\]
we obtain
\[
  \Jh T = \Jh(\mathcal A_{11}) + \Jh(\mathcal A_{22}) + \Jh(\mathcal A_{33}) = \Jh(\mathcal A_{22}) + \Jh(\mathcal A_{33}) + 2T.
\]
This proves \eqref{eq:moving-trace-correction}.

It remains to prove the simple-eigenvalue correction.  Assume $\lambda_M>\lambda_m$. We first record the elementary mixed-derivative identities
\begin{equation}
  \label{eq:mixed-derivative-commutation} (\mathcal A_{23})_2=(\mathcal A_{22})_3, \qquad (\mathcal A_{23})_3=(\mathcal A_{33})_2.
\end{equation}
Indeed, since $r_2=r_3=0$ at $e_1$,
\[
  (\mathcal A_{23})_2 = \partial_2(rv_{23}) = v_{232}, \qquad (\mathcal A_{22})_3 = \partial_3(rv_{22}) = v_{223},
\]
so they are equal.  The second identity is the same.

Now compute the second derivatives of the largest eigenvalue.  For $k\in\{2,3\}$, set
\[
  y_k(s):=e_1+s e_k, \qquad R(s):=1+s^2.
\]
If $k=2$, then an orthonormal basis of $y_2(s)^\perp$ is
\[
  \tau_2(s):=\frac{e_2-se_1}{\sqrt{R(s)}}, \qquad e_3.
\]
In this basis, using \eqref{eq:zero-homogeneous-identities}, the matrix of $\HH|_{y_2(s)^\perp}$ is
\[
M_2(s)
        = \begin{pmatrix}
        R(s)\mathcal A_{22}(y_2(s))
        &
        \sqrt{R(s)}\mathcal A_{23}(y_2(s))
        \\
        \sqrt{R(s)}\mathcal A_{23}(y_2(s))
        &
        \mathcal A_{33}(y_2(s))
        \end{pmatrix}.
\]
Similarly, if $k=3$, then an orthonormal basis of $y_3(s)^\perp$ is
\[
  e_2, \qquad \tau_3(s):=\frac{e_3-se_1}{\sqrt{R(s)}},
\]
and the corresponding matrix is
\[
M_3(s)
        = \begin{pmatrix}
        \mathcal A_{22}(y_3(s))
        &
        \sqrt{R(s)}\mathcal A_{23}(y_3(s))
        \\
        \sqrt{R(s)}\mathcal A_{23}(y_3(s))
        &
        R(s)\mathcal A_{33}(y_3(s))
        \end{pmatrix}.
\]
At $s=0$, both matrices are
\[
\begin{pmatrix}
        \lambda_M & 0\\
        0 & \lambda_m
        \end{pmatrix},
        \qquad \lambda_M>\lambda_m.
\]
For a $C^2$ family of symmetric $2\times2$ matrices
\[
M(s)=
        \begin{pmatrix}
        A(s)&C(s)\\
        C(s)&B(s)
        \end{pmatrix},
        \qquad
        M(0)=
        \begin{pmatrix}
        \lambda_M&0\\
        0&\lambda_m
        \end{pmatrix},
        \qquad \lambda_M>\lambda_m,
\]
the eigenvalue branch $\lambda(s)$ with $\lambda(0)=\lambda_M$ satisfies (see, for instance, \cite[Eq. (1.73)]{Tao2012RandomMatrixTheory})
\[
  \lambda'(0)=A'(0), \qquad \lambda''(0)=A''(0)+\frac{2C'(0)^2}{\lambda_M-\lambda_m}.
\]
(For the lowest eigenvalue, the addition in $\lambda''(0)$ becomes a subtraction.) Applying this formula to $M_2$ and $M_3$, and using
\[
  R(0)=1, \qquad R'(0)=0, \qquad R''(0)=2, \qquad \mathcal A_{23}(e_1)=0,
\]
we get
\[
  (\lambda_M)_2=(\mathcal A_{22})_2, \qquad (\lambda_M)_{22} = (\mathcal A_{22})_{22} + 2\lambda_M + \frac{2(\mathcal A_{23})_2^2}{\lambda_M-\lambda_m},
\]
and
\[
  (\lambda_M)_3=(\mathcal A_{22})_3, \qquad (\lambda_M)_{33} = (\mathcal A_{22})_{33} + \frac{2(\mathcal A_{23})_3^2}{\lambda_M-\lambda_m}.
\]
Therefore, from Lemma~\ref{lem:tensor-identities},
\[
\begin{split}
        \Jh \lambda_M
        &= (\lambda_M)_{22}+(\lambda_M)_{33}
        + \apalpha\frac{v_2(\lambda_M)_2+v_3(\lambda_M)_3}{v}
        + \frac{T}{v}\lambda_M  \\
        &= \Jh(\mathcal A_{22})
        + 2\lambda_M
        + \frac{2}{\lambda_M-\lambda_m}
        \left((\mathcal A_{23})_2^2+(\mathcal A_{23})_3^2\right).
\end{split}
\]
Using the mixed-derivative identities proved above,
\[
  (\mathcal A_{23})_2=(\mathcal A_{22})_3, \qquad (\mathcal A_{23})_3=(\mathcal A_{33})_2,
\]
we obtain \eqref{eq:moving-eigenvalue-correction-max}. A symmetric argument gives \eqref{eq:moving-eigenvalue-correction-min}.
\end{proof}

The first consequence is an interior positivity and boundary-limit statement.

\begin{lem}
\label{lem:interior-trace-zero-flatness}
Let $n=3$, $\gamma\in(0,1)$.  Let $u$ be a classical cone for the $\gamma$-Alt--Phillips problem in $\R^3$, and let $v$ denote the transformed one-homogeneous solution.  Let $\mathcal{V}$ be a nonflat connected component of $\{v>0\}$, and set
\[
  \mathcal{V}_S:=\mathcal{V}\cap\Sn .
\]
Then, recalling Definition~\ref{def:homogeneous-hessian},
\[
  T = r\Delta v = \frac{\apalpha r}{2}\frac{1-|\grad v|^2}{v} >0, \qquad \lambda_M>0 \qquad\text{in}\quad \overline{\mathcal{V}}.
\]
Moreover, if $h$ denotes the inward geodesic curvature of $\partial\mathcal{V}_S$, then
\[
  h>0 \qquad\text{on }\partial\mathcal{V}_S .
\]
Finally, as $\mathcal{V}_S\ni\omega\to q\in\partial\mathcal{V}_S$,
\[
  \lambda_M(\omega)\to h(q), \qquad \lambda_m(\omega)\to-\frac{h(q)}{1+\apalpha}, \qquad \frac{T(\omega)}{\lambda_M(\omega)} \to \frac{\apalpha}{1+\apalpha}.
\]
\end{lem}

\begin{proof}
Set $P=|\grad v|^2$ and $S=1-P$.  By Proposition~\ref{prop:positive-mean-curvature} (see also Remark~\ref{rem:positive-mean-curvature-cones}), $S>0$ in $\mathcal{V}$ and thus $ T > 0$ as well, which by definition gives the first result. The fact that $h > 0$ comes from Proposition~\ref{prop:positive-mean-curvature} as well, and 1-homogeneity.

Let us then compute the boundary limits of the two nonradial eigenvalues of $\HH=rD^2v$.  Fix $q\in\partial\mathcal{V}_S$.  After a rotation, assume $q = e_1$, $e_2$ is tangent to $\partial\mathcal{V}_S$ at $e_1$, and $e_3$ is the inward unit normal to $\mathcal{V}_S$.
All derivatives below are ordinary Euclidean derivatives evaluated at $e_1$.  We know
\[
  v(e_1)=0, \qquad \grad v(e_1)=e_3, \qquad |\grad v|=1 \quad\text{on }\partial\mathcal{V} .
\]
Let $\gamma(s)\subset\partial\mathcal{V}_S$ be a unit-speed parametrization with
\[
  \gamma(0)=e_1, \qquad \gamma'(0)=e_2.
\]
With our sign convention for the inward geodesic curvature,
$
  \gamma''(0)\cdot e_3=-h(e_1) .
$
Since $v=0$ on the free boundary,
$
  v(\gamma(s))=0 .
$
Differentiating twice at $s=0$, we get
\begin{equation}\label{eq:boundary-v22-h}
        v_{22}(e_1)=h(e_1).
\end{equation}

Next, differentiating the boundary condition $|\grad v|^2=1$ along $\gamma$, we obtain
\begin{equation}\label{eq:boundary-v23-zero}
        v_{23}(e_1)=0 .
\end{equation}

It remains to consider $v_{33}(e_1)$. By \cite[Lemma 5.1]{KarakhanyanSanzPerela2026StableConesAltPhillips} we directly know $v_{33}(e_1) = d = -\frac{h(e_1)}{1+\apalpha}$. We still deduce it for completeness: 

Since $v$ is one-homogeneous, $v_{11}(e_1)=0 $.
Set
\[
  d:=v_{33}(e_1).
\]
 Taylor expansion gives
\[
  v(e_1+te_3) = t+\frac12dt^2+o(t^2),
\qquad\text{and}\qquad 
  |\grad v(e_1+te_3)|^2 = 1+2dt+o(t).
\]
Hence
\[
  \frac{1-|\grad v(e_1+te_3)|^2}{v(e_1+te_3)} \to -2d \qquad\text{as }t\downarrow0.
\]
Taking the limit in
\[
  \Delta v = \frac{\apalpha}{2} \frac{1-|\grad v|^2}{v}
\]
along $e_1+te_3$, and using
$
  \Delta v(e_1) = v_{11}(e_1)+v_{22}(e_1)+v_{33}(e_1) = h(e_1)+d,
$
we obtain
$
  h(e_1)+d = -\apalpha d 
$, and thus
\begin{equation}\label{eq:boundary-v33}
        v_{33}(e_1)
        = d
        = -\frac{h(e_1)}{1+\apalpha}.
\end{equation}

Combining \eqref{eq:boundary-v22-h}, \eqref{eq:boundary-v23-zero}, and \eqref{eq:boundary-v33}, the matrix of $D^2v(e_1)$ on $e_1^\perp=\operatorname{span}\{e_2,e_3\}$ is
\[
\begin{pmatrix}
        v_{22} & v_{23} \\
        v_{23} & v_{33}
        \end{pmatrix}
        = \begin{pmatrix}
        h(e_1) & 0 \\
        0 & -\dfrac{h(e_1)}{1+\apalpha}
        \end{pmatrix}.
\]
Since $|e_1|=1$, we also have
\[
  \HH(e_1)|_{e_1^\perp} = D^2v(e_1)|_{e_1^\perp}.
\]
By continuity of $\HH$ up to the regular free boundary and continuity of the plane $\omega^\perp$, the eigenvalues of $\HH(\omega)|_{\omega^\perp}$ converge to the eigenvalues of this matrix as $\mathcal{V}_S\ni\omega\to e_1$.  Since $h(e_1)>0$, the larger and smaller eigenvalues satisfy
\[
  \lambda_M(\omega)\to h(e_1), \qquad \lambda_m(\omega)\to-\frac{h(e_1)}{1+\apalpha}.
\]
Therefore
\[
  T(\omega)=\lambda_M(\omega)+\lambda_m(\omega) \to h(e_1)-\frac{h(e_1)}{1+\apalpha} = \frac{\apalpha}{1+\apalpha}h(e_1),
\]
and hence
\[
  \frac{T(\omega)}{\lambda_M(\omega)} \to \frac{\apalpha}{1+\apalpha}.
\]
Since the boundary point $q\in\partial\mathcal{V}_S$ was arbitrary, the result follows.
\end{proof}

In the following, we sharpen the lower bound on $T$. The methods used here are inspired by \cite{Hamilton1982ThreeManifoldsPositiveRicciCurvature, HuWeiYangZhou2024AlexandrovFenchelCapillary}.

\begin{prop}[Pinching inequality]
\label{prop:pinching-inequality}
In the setting of Lemma \ref{lem:interior-trace-zero-flatness} we have
\[
  \lambda_M+(1+\apalpha)\lambda_m>0\qquad\text{in}\quad \mathcal V_S.
\]
Equivalently,
\[
  T> \frac{\apalpha}{1+\apalpha}\,\lambda_M \qquad\text{in}\quad \mathcal V_S.
\]
\end{prop}

\begin{proof}
Recall $\mathcal V$ is nonflat and put $c=\apalpha/(1+\apalpha)$.  Suppose, toward a contradiction for the non-strict inequalities first, that
\[
  U:=\{x\in \mathcal{V}:T(x)<c\,\lambda_M(x)\}
\]
is nonempty.  All quantities are homogeneous of degree zero, so we work on the spherical link $\mathcal{V}_S = \mathcal{V}\cap\Sn$.  By Lemma~\ref{lem:interior-trace-zero-flatness}, $T>0$ and $\lambda_M>0$ at all interior points.  On $U$, set $z=T/\lambda_M$.  Since $z<c<1$,
\[
  \lambda_m=T-\lambda_M=\lambda_M(z-1)<0<\lambda_M.
\]
Thus $\lambda_M$ is a simple eigenvalue and $z$ is smooth on $U$.

Since \(z\to c\) along \(\partial U\), while \(z<c\) in \(U\), the function \(z\) achieves its minimum at an interior point of \(U\). At this point $0<z<c$. Let
\[
  \Lin_{h,0}:=\Jh-r^2Q.
\]
At the minimum, $\Lin_{h,0}z\ge0$.  After a rotation, we may assume that the minimum point is $e_1$, and that the tangent basis $e_2,e_3$ diagonalizes $\HH|_{e_1^\perp}$:
\[
  A_{22}=\lambda_M, \qquad A_{33}=\lambda_m, \qquad A_{23}=0.
\]
Set
\[
  d=\lambda_M-\lambda_m=\lambda_M(2-z), \quad \tau_1=T+\apalpha \lambda_M=\lambda_M(z+\apalpha), \quad \tau_2=T+\apalpha \lambda_m=\lambda_M((1+\apalpha)z-\apalpha).
\]
Then $\tau_2<0$.  Since $\grad z=0$ at the minimum,
\[
  \lambda_M\Lin_{h,0}z=\Jh T-z\Jh \lambda_M .
\]
Let us rename  the variables to act as coefficients in the following computations: 
\[
  \xi=v_2, \qquad \eta=v_3, \qquad n=(\mathcal A_{22})_3, \qquad \ell=(\mathcal A_{33})_2.
\]
By Lemma~\ref{lem:tensor-identities} and the moving-plane correction \eqref{eq:moving-trace-correction},
\[
  \Jh T =(2+\apalpha)T-\frac{\apalpha}{v}(\lambda_M^2+\lambda_m^2) +\frac{2}{v^2}(\tau_1\xi^2+\tau_2\eta^2).
\]
By Lemma~\ref{lem:tensor-identities} and \eqref{eq:moving-eigenvalue-correction-max},
\begin{equation}
  \label{eq:jh-lambda-max} \Jh \lambda_M =(2+\apalpha)\lambda_M-\frac{\apalpha}{v}\lambda_M^2 +\frac{2\tau_1}{v^2}\xi^2 +\frac{2}{d}(n^2+\ell^2).
\end{equation}
Subtracting $z\Jh \lambda_M$, and using $\lambda_m=\lambda_M(z-1)$, gives
\begin{equation}\label{eq:pinching-intermediate}
  \lambda_M\Lin_{h,0}z
  =\frac{\apalpha \lambda_m d}{v}
  +\frac{2}{v^2}\bigl((1-z)\tau_1\xi^2+\tau_2\eta^2\bigr)
  -\frac{2z}{d}(n^2+\ell^2).
\end{equation}
It remains only to eliminate $\ell$.  From $\grad T=z\grad \lambda_M$ and Lemma~\ref{lem:tensor-identities},
\[
  T_2=-\frac{\tau_1}{v}\xi, \qquad (\lambda_M)_2=\frac{T_2}{z}=-\frac{\tau_1}{zv}\xi.
\]
Since $T_2=(\lambda_M)_2+\ell$ (by Lemma~\ref{lem:moving-plane-corrections}),
\[
  \ell=T_2-(\lambda_M)_2=\frac{1-z}{z}\frac{\tau_1}{v}\xi .
\]
Substituting this into \eqref{eq:pinching-intermediate} gives
\[
\begin{split}
  \lambda_M\Lin_{h,0}z
  &=\frac{\apalpha \lambda_M^2(z-1)(2-z)}{v}  +
  \frac{2\tau_2}{v^2}
  \left[
    \eta^2+
    \frac{(1-z)(z+\apalpha)}{z(2-z)}\xi^2
  \right]
  -\frac{2z}{\lambda_M(2-z)}n^2 .
\end{split}
\]
Because $0<z<c=\apalpha/(1+\apalpha)$, we have $z-1<0$ and $\tau_2<0$.  The bracket is nonnegative, the last term is nonpositive, and the first term is strictly negative.  Hence $\lambda_M\Lin_{h,0}z<0$, contradicting $\Lin_{h,0}z\ge0$.  Thus $U$ is empty, and the pinching inequality in the non-strict case follows.

 Finally, if equality were to occur in the interior, by the previous argument, it would be at a local minimum, and the same computation as before yields a contradiction again. Thus, the inequality is strict. 
\end{proof}

\section{Properties of the test function}
\label{sec:properties-test-functions}
We now turn the pinching estimate into a differential inequality for the highest eigenvalue, which informs the choice of test function. 

\begin{lem}
\label{lem:simple-largest-eigenvalue-inequality}
In the setting of Lemma \ref{lem:interior-trace-zero-flatness} and for some
\[
  \frac{\apalpha}{2+\apalpha}\le p\le\frac{1}{2+\apalpha},
\]
we have that
\[
  \Jh(\lambda_M^p)\ge p(2+\apalpha)\lambda_M^p
\]
holds classically in  $\{v > 0\} $ and $ \lambda_M > \lambda_m$.
\end{lem}

\begin{proof}
By scale invariance, rotate and dilate so that the point is $e_1$, and choose $e_2,e_3$ so that
\[
  A_{22}=\lambda_M, \qquad A_{33}=\lambda_m, \qquad A_{23}=0.
\]
Put $d=\lambda_M-\lambda_m$, and write, as in Proposition~\ref{prop:pinching-inequality}
\[
  m=(\mathcal A_{22})_2, \qquad n=(\mathcal A_{22})_3, \qquad \ell=(\mathcal A_{33})_2, \qquad \tau_1=T+\apalpha \lambda_M, \qquad X=\frac{\tau_1v_2}{v}.
\]
For $\lambda_M^p$, the zero-order term in $\Jh$ changes the usual chain rule to
\begin{equation}
  \label{eq:jh-lambda-max-chain-rule} \Jh(\lambda_M^p) =p \lambda_M^{p-1}\Jh \lambda_M+r^2p(p-1)\lambda_M^{p-2}|\grad \lambda_M|^2 +(1-p)r^2Q\lambda_M^p .
\end{equation}
Again, as in the proof of Lemma~\ref{lem:interior-trace-zero-flatness}
\[
  m+\ell=T_2=-\frac{1}{v}(\apalpha \lambda_M v_2+Tv_2)=-X, \qquad \ell=-X-m.
\]
Since $\lambda_M>\lambda_m$ and $ \lambda_M > 0$, the largest eigenvalue is smooth and simple near the point. In a local orthonormal frame which diagonalizes $B=A|_{e_1^\perp}$ at $e_1$, the first-variation formula for a simple eigenvalue gives $ \nabla_k \lambda_M=(\nabla_k B)_{22} = (\nabla_k \mathcal{A})_{22}. $ Thus, with $m$ and $n$ as above (recall \eqref{eq:mixed-derivative-commutation}) we have
\[
  |\nabla \lambda_M|^2=m^2+n^2.
\]
(Notice that the variation of the eigendirection appears only in the second derivative.)

We have
\[
\begin{split}
\frac{\Jh(\lambda_M^p)-p(2+\apalpha)\lambda_M^p}{\lambda_M^p}
&= \frac{(1-p)T-p\apalpha \lambda_M}{v}
+\frac{2pX^2}{\lambda_M\tau_1}
+\frac{2p}{\lambda_M d}(n^2+\ell^2)
-\frac{p(1-p)}{\lambda_M^2}(m^2+n^2).
\end{split}
\]
Substituting $\ell=-X-m$ and grouping terms gives
\[
\begin{split}
\frac{\Jh(\lambda_M^p)-p(2+\apalpha)\lambda_M^p}{\lambda_M^p}
&= \frac{(1-p)T-p\apalpha \lambda_M}{v}
+\left(\frac{2p}{\lambda_M d}-\frac{p(1-p)}{\lambda_M^2}\right)n^2  \\
&\quad
+\frac{2p}{\lambda_M d}(m+X)^2
-\frac{p(1-p)}{\lambda_M^2}m^2
+\frac{2pX^2}{\lambda_M\tau_1}.
\end{split}
\]
Set
\[
  K_p=(1+p)\lambda_M+(1-p)\lambda_m=T+p(\lambda_M-\lambda_m).
\]
Since $d=\lambda_M-\lambda_m$, we have $ 2\lambda_M-(1-p)d  =     (1+p)\lambda_M+(1-p)\lambda_m   =    K_p, $ and thus the coefficient for $n^2$ is $\tfrac{pK_p}{\lambda_M^2 d}$. The same occurs for the $m^2$-term:
\[
\begin{split}
\frac{2p}{\lambda_M d}(m+X)^2
-\frac{p(1-p)}{\lambda_M^2}m^2
&= \frac{pK_p}{\lambda_M^2 d}m^2
+\frac{4p}{\lambda_M d}mX
+\frac{2p}{\lambda_M d}X^2 .
\end{split}
\]
Completing the square in $m$ gives
\[
\begin{split}
\frac{pK_p}{\lambda_M^2 d}m^2
+\frac{4p}{\lambda_M d}mX
&= \frac{pK_p}{\lambda_M^2 d}
\left(m+\frac{2\lambda_M X}{K_p}\right)^2
-\frac{4pX^2}{dK_p}.
\end{split}
\]
Hence the remaining $X^2$-coefficient is
\[
  \frac{2p}{\lambda_M d} +\frac{2p}{\lambda_M\tau_1} -\frac{4p}{dK_p}.
\]
Observe that
\[
\begin{split}
\frac1{\lambda_M d}+\frac1{\lambda_M\tau_1}-\frac2{dK_p}
&= \frac1\lambda_M\left(\frac1{\tau_1}
-\frac{1-p}{K_p}\right) =
\frac{K_p-(1-p)\tau_1}{\lambda_M\tau_1K_p}.
\end{split}
\]
Using $K_p=T+pd$, $\tau_1=T+\apalpha \lambda_M$, and $T+d=2\lambda_M$, we get
\[
\begin{split}
K_p-(1-p)\tau_1
 = \lambda_M\bigl(p(2+\apalpha)-\apalpha\bigr).
\end{split}
\]
Combining the previous computations yields
\begin{equation}
\label{eq:simple-no-discard}
\begin{split}
\frac{\Jh(\lambda_M^p)-p(2+\apalpha)\lambda_M^p}{\lambda_M^p}
&=\frac{(1-p)T-p\apalpha \lambda_M}{v}
  +\frac{pK_p}{\lambda_M^2 d}n^2 \\
&\quad
  +\frac{pK_p}{\lambda_M^2 d}
   \left(m+\frac{2\lambda_M X}{K_p}\right)^2
  +\frac{2pX^2}{\tau_1K_p}
    \bigl(p(2+\apalpha)-\apalpha\bigr).
\end{split}
\end{equation}
We want to obtain a nonnegative sign on the right-hand side. The pinching inequality, Proposition~\ref{prop:pinching-inequality}, gives $T\ge \apalpha \lambda_M/(1+\apalpha)>0$.  We also know $K_p>0$,  $d>0$, and $\tau_1=T+\apalpha \lambda_M>0$.  Moreover,
\[
  (1-p)T-p\apalpha \lambda_M \ge \frac{\apalpha \lambda_M}{1+\apalpha}\bigl(1-p(2+\apalpha)\bigr) \ge0
\]
when $p\le1/(2+\apalpha)$. Finally,  the last term $p(2+\apalpha)-\apalpha$ is nonnegative when $p\ge\apalpha/(2+\apalpha)$.  Thus every term on the right-hand side is nonnegative, which proves the lemma.
\end{proof}

The next lemma removes the simple-eigenvalue restriction.

\begin{lem} 
\label{lem:largest-eigenvalue-no-negative-defect}
In the setting of Lemma \ref{lem:interior-trace-zero-flatness}, let $F=\lambda_M^p$ with
\[
  \frac{\apalpha}{2+\apalpha}\le p\le\frac{1}{2+\apalpha}.
\]
Then,
\[
  \Jh F-p(2+\apalpha)F\ge0
\]
holds locally in the viscosity sense, and hence locally in the distributional sense, in  $\Om\setminus\{0\}$.
\end{lem}

\begin{proof}
On a flat component $\HH\equiv0$, so $F\equiv0$.  We therefore work on the slice $r=1$ of one connected nonflat component.

Thanks to Lemma~\ref{lem:simple-largest-eigenvalue-inequality}, it only remains to consider a nonsimple point $x_0$ (with $\lambda_M = \lambda_m$), which, after a rotation, we assume is $e_1 = (1, 0, 0)$. In the two-dimensional nonradial plane,
\[
  \HH|_{e_1^\perp}=\lambda_M g_{\mathbb{S}^2}, \qquad T=2\lambda_M, \qquad \lambda_M>0 .
\]
Let $\varphi\in C^2$ touch $F=\lambda_M^p$ from above at $e_1$.  After shrinking the neighborhood we may assume $\varphi>0$, and we set
\[
  \sigma=\varphi^{1/p}.
\]
Then $\sigma\ge\lambda_{\max}(\HH|_{x^\perp})$, with equality at $e_1$. Since the nonradial plane is two-dimensional,
\[
  T\le 2\lambda_{\max}(\HH|_{x^\perp})\le 2\sigma,
\]
and equality holds at $e_1$.  Hence
\[
  w:=2\sigma-T
\]
has a local minimum zero at $e_1$.  With
\[
  \Lin_{h,0}:=\Jh-r^2Q,
\]
the minimum principle gives
\[
  0\le \Lin_{h,0}w(e_1) =2\Jh\sigma(e_1)-\Jh T(e_1),
\]
because $w(e_1)=0$.  Thus
\[
  \Jh\sigma(e_1)\ge \frac12\Jh T(e_1).
\]
At the nonsimple point, $\HH|_{x^\perp}=\lambda_M g$.  By \eqref{eq:moving-trace-correction} and Lemma~\ref{lem:tensor-identities},
\[
\begin{split}
  \Jh T
  &=\Jh(\mathcal A_{22})+\Jh(\mathcal A_{33})+2T   =2(2+\apalpha)\lambda_M-\frac{2\apalpha}{v}\lambda_M^2
    +\frac{2(2+\apalpha)\lambda_M}{v^2}|\grad_\perp v|^2 ,
\end{split}
\]
where $|\nabla_\perp v|^2 = v_2^2 + v_3^2$. Therefore
\[
  \Jh\sigma(e_1) \ge (2+\apalpha)\lambda_M-\frac{\apalpha}{v}\lambda_M^2 +\frac{(2+\apalpha)\lambda_M}{v^2}|\grad_\perp v|^2 .
\]
Moreover, since $2\sigma-T$ has a minimum zero at $e_1$,
\[
  2\grad\sigma=\grad T .
\]
The trace-gradient identity in Lemma~\ref{lem:tensor-identities} gives
\[
  \grad T =-\frac1v\bigl(\apalpha\HH(\grad_\perp v,\cdot) +T\grad_\perp v\bigr) =-\frac{(2+\apalpha)\lambda_M}{v}\grad_\perp v,
\]
and hence
\[
  |\grad\sigma|^2 = \frac{(2+\apalpha)^2\lambda_M^2}{4v^2}|\grad_\perp v|^2 .
\]
Using the chain rule for $\varphi=\sigma^p$ (as in \eqref{eq:jh-lambda-max-chain-rule}), and using $ Q=T/v=2\lambda_M/v$ at $e_1$, we obtain
\begin{equation}
\label{eq:nonsimple-no-discard}
\begin{split}
\frac{(\Jh-p(2+\apalpha))\varphi}{\lambda_M^p}(e_1)
&\ge
\frac{(2-p(2+\apalpha))\lambda_M}{v}   +
\frac{p(2+\apalpha)}{v^2}
\left(
  1-\frac{(1-p)(2+\apalpha)}4
\right)
|\grad_\perp v|^2 .
\end{split}
\end{equation}
The first term is nonnegative because $p\le 1/(2+\apalpha)$.  The second coefficient is nonnegative because
\[
  p\ge \frac{\apalpha}{2+\apalpha} \quad\Longrightarrow\quad (1-p)(2+\apalpha)\le 2.
\]
Thus, the viscosity inequality also holds at nonsimple points as well. Since the operator is locally smooth and the function being tested is Lipschitz, standard arguments yield the distributional inequality as well (see, e.g., \cite{Ishii1995WeakViscosityDistributionSolutions}).
\end{proof}
\begin{remark}[Endpoint case]
\label{rem:endpoint-case}
In the case $\gamma = 2/3$ (so $\apalpha = 1$ and $p = 1/3$), set $F=\lambda_M^{1/3}$. It will be useful to write the more refined inequality
\begin{equation}\label{eq:endpoint-largest-defect-with-R}
  \Jh F-F\ge \mathcal R
  \qquad\text{in }\mathcal V,
\end{equation}
in the distributional sense, where the degree-zero function $\mathcal R$ is given by
\begin{equation}\label{eq:endpoint-largest-R-definition}
  \mathcal R:=\frac{F}{3\psi}(2T-\lambda_M)>0\qquad\text{on}\quad \mathcal{V}_S,
\end{equation}
thanks to Proposition~\ref{prop:pinching-inequality} (recall $ v = r\psi$). 
This is a consequence of the proofs of 
Lemmas~\ref{lem:simple-largest-eigenvalue-inequality} and \ref{lem:largest-eigenvalue-no-negative-defect} by not discarding the first term in the inequalities \eqref{eq:simple-no-discard}-\eqref{eq:nonsimple-no-discard}.  
\end{remark}

This gives the homogeneous test function needed for the stability reduction.

\begin{prop} 
\label{prop:largest-eigenvalue-test}
In the setting of Lemma \ref{lem:interior-trace-zero-flatness}, let $\Lambda=\lamax(D^2v)=\lambda_M/r$ and $G=\Lambda^p$.  If
\[
  \frac{\apalpha}{2+\apalpha}\le p\le\frac{1}{2+\apalpha},
\]
then we have that
\[
  \Lin G\ge p(1+p)\frac{G}{r^2}
\]
holds locally in the distributional sense in $\Om\setminus\{0\}$.
\end{prop}

\begin{proof}
On flat components the claim is trivial.  On a nonflat component, set $F=\lambda_M^p$.  By Lemma~\ref{lem:largest-eigenvalue-no-negative-defect},
\[
  \Jh F\ge p(2+\apalpha)F
\]
locally in the distributional sense.  For a smooth degree-zero function $F$, $x\cdot \nabla F = 0$ and hence
\[
  \Lin(r^{-p}F) =r^{-p-2}\{\Jh F-p(1+\apalpha-p)F\}.
\]
Applying this identity distributionally to $F=\lambda_M^p$, by a standard local approximation gives the desired result.
\end{proof}

We finally combine this test with the Hardy obstruction.

\begin{prop} 
\label{prop:nonflat-stability-contradiction}
Let $n=3$, $\gamma\in(0,2/3)$. Every classical stable cone satisfies $\HH\equiv0$.
\end{prop}

\begin{proof}
If $\Omega = \R^3 \setminus\{0\}$, it follows from Proposition~\ref{prop:full-support-exclusion}. Suppose then $\partial\Omega_S\neq\varnothing$. Since $0<\apalpha<1$,
\[
  \lambda_\apalpha=\frac{(1+\apalpha)^2}{4}<1.
\]
Choose (recalling the condition on $p$ from Lemma~\ref{lem:largest-eigenvalue-no-negative-defect})
\[
  \frac{\apalpha}{2+\apalpha} < \frac{\lambda_\apalpha}{2+\apalpha}<p<\frac{1}{2+\apalpha}.
\]
Assume, toward a contradiction, that $\HH\not\equiv0$, and choose a nonflat connected component.  By Lemma~\ref{lem:interior-trace-zero-flatness}, $\lambda_M>0$ on that component (up to the boundary), so $\Lambda=\lambda_M/r$ is nonzero there.  Set $G=\Lambda^p$.  By Proposition~\ref{prop:largest-eigenvalue-test},
\[
  \Lin G\ge p(1+p)\frac{G}{r^2}
\]
holds in the distributional sense. But then, Proposition~\ref{prop:stability-hardy-reduction} (notice that $G$ is positive and  Lipschitz up to the boundary) forces
\[
  p(1+p)\le\lambda_{\apalpha,p}.
\]
This is a contradiction, since
\[
  p(1+p)>\lambda_{\apalpha,p} \quad\Longleftrightarrow\quad p(2+\apalpha)>\lambda_\apalpha.
\]
Hence, no nonflat component exists, and $\HH\equiv0$. Since $v$ is a classical cone, it must be a half-space solution.
\end{proof}

Let us now deal with the endpoint case, corresponding to $\gamma = 2/3$, equivalently $\alpha = 1$.

\begin{prop} 
\label{prop:endpoint-largest-eigenvalue-contradiction}
Let $n=3$ and $\gamma= 2/3$. Every classical stable cone satisfies $\HH\equiv0$.
\end{prop}

\begin{proof}
If $\Omega = \R^3 \setminus\{0\}$, it follows from Proposition~\ref{prop:full-support-exclusion}. Suppose then $\partial \Omega_S\neq\varnothing$. 

Let us argue by contradiction. Assume it is not true, and choose a nonflat connected component $\mathcal V$. Let us denote 
\[
  F:=\lambda_M^{1/3},
  \qquad
  G:=r^{-1/3}F=\left(\frac{\lambda_M}{r}\right)^{1/3}=\Lambda^{1/3}
\]
in $\Omega$. As in Proposition~\ref{prop:largest-eigenvalue-test} and thanks to Remark~\ref{rem:endpoint-case}, locally on the cone over $\mathcal V_S$,
\begin{equation}\label{eq:endpoint-largest-LG-with-R}
  \Lin G
  \ge
  \frac49\frac{G}{r^2}+r^{-7/3}\mathcal R(\omega)\ge \frac49\frac{G}{r^2}+r^{-7/3}\vartheta(\omega)\mathcal R(\omega)
\end{equation}
distributionally in $\Omega$,  where we have chosen some $\vartheta\in C_c^\infty(\mathcal V_S)$ with $0\le\vartheta\le1$ and
$\vartheta\not\equiv0$, and $\mathcal R$ is given by \eqref{eq:endpoint-largest-R-definition}. In particular, since $\mathcal R>0$, $F>0$,
and $\psi>0$ on $\supp\vartheta$ (recall $v = r\psi$), 
\begin{equation}\label{eq:endpoint-largest-Btheta}
  B_\vartheta:=\int_{\mathcal V_S}\psi F\vartheta\mathcal R\,d\omega> 0. 
\end{equation}

We now repeat the stability computation of Step 2 in the proof of 
Proposition~\ref{prop:stability-hardy-reduction}, using
\eqref{eq:endpoint-largest-LG-with-R}.   For the additional term, no boundary limit has
to be estimated: on every fixed annulus the cone over $\supp\vartheta$ stays a
positive distance from $\partial\Omega$.  We obtain
\begin{equation}\label{eq:endpoint-before-radial-hardy-largest}
  \frac49A_G\int_0^\infty r^{1/3}\eta^2\,dr
  +B_\vartheta\int_0^\infty r^{1/3}\eta^2\,dr
  \le
  A_G\int_0^\infty r^{7/3}|\eta'|^2\,dr,
\end{equation}
where
$
  A_G:=\int_{\Omega_S}\psi F^2\,d\omega>0
$.

The sharp one-dimensional Hardy constant for these radial weights is
\begin{equation}\label{eq:endpoint-radial-hardy-largest}
  \inf_{\eta\in C_c^\infty((0,\infty))\setminus\{0\}}
  \frac{\displaystyle\int_0^\infty r^{7/3}|\eta'|^2\,dr}
       {\displaystyle\int_0^\infty r^{1/3}\eta^2\,dr}
  =\frac49 .
\end{equation}
Taking a minimizing sequence in \eqref{eq:endpoint-before-radial-hardy-largest} gives
\[
  \frac49A_G+B_\vartheta\le \frac49A_G,
\]
contradicting $B_\vartheta>0$.  Hence, no nonflat component exists, and
$\HH\equiv0$.
\end{proof}

We finally have:
\begin{proof}[Proof of Theorem~\ref{thm:stable-homogeneous-cone-flatness}]
It follows from Proposition~\ref{prop:nonflat-stability-contradiction} or Proposition~\ref{prop:endpoint-largest-eigenvalue-contradiction}.
\end{proof}

As well as:
\begin{proof}[Proof of Corollary~\ref{cor:free-boundary-regularity}]
This is a classical consequence of   Theorem~\ref{thm:stable-homogeneous-cone-flatness} (see \cite{Phillips1983MinimizationFreeBoundary, AltPhillips1986FreeBoundarySemilinear, DeSilvaSavin2021DegenerateOnePhase,Bonorino2001RegularityFreeBoundaryII, RestrepoRosOton2025CinftyRegularitySemilinear}). The bound on the size of the singular set is given in \cite[Appendix A]{FernandezRealYu2023GenericPropertiesFreeBoundary}.
\end{proof}

\bibliographystyle{abbrv}
\bibliography{flat_alt_phillips_refs}

\end{document}